\documentclass{cmslatex}
\usepackage[paperwidth=7in, paperheight=10in, margin=.9in]{geometry}
 \usepackage[backref,colorlinks,linkcolor=red,anchorcolor=green,citecolor=blue]{hyperref}
\usepackage{amsfonts,amssymb}
\usepackage{amsmath}
\usepackage{graphicx}
\usepackage{cite}
\usepackage{enumerate}
\renewcommand{\theequation}{\arabic{section}.\arabic{equation}}

   \allowdisplaybreaks
\begin{document}
 \title{Analysis of Hierarchical ensemble Kalman Inversion}


          \author{Neil K. chada\thanks{Mathematics Institute, University of Warwick, Coventry, CV4 7AL, UK, (n.chada@warwick.ac.uk).}}

         \pagestyle{myheadings} \markboth{}{N. K. Chada} \maketitle

 \begin{abstract}
 We discuss properties of hierarchical Bayesian inversion through the ensemble Kalman filter (EnKF). Our focus will be primarily on deriving continuous-time limits for hierarchical  inversion in the linear case. An important characteristic of the EnKF for inverse problems is that the updated particles are preserved by the linear span of the initial ensemble. By incorporating certain hierarchical approaches we show that we can break away from the induced {subspace property}. We further consider a number of variants of the EnKF such as covariance inflation and localization, where we derive their continuous-time limits. We verify these results with various numerical experiments through a linear elliptic partial differential equation. 
\end{abstract}

\section{Introduction}
The ensemble Kalman filter (EnKF) \cite{GE09,EL96} was proposed by Evensen in 1994 as a Monte-Carlo approximation of the Kalman filter (KF). Its motivation was based on mitigating the computational challenges associated with the KF, replacing the updated mean and covariances with an ensemble of particles. Since then the EnKF has been widely applied in numerous fields such as weather prediction and oceanography \cite{JLA02, GE94, LO05}. Given its robustness and Bayesian formulation paradigm, the EnKF has been further applied to inverse problems. 
Inverse problems are concerned with the recovery of some quantity of interest $u \in \mathcal{X}$ from noisy measurements $y \in \mathcal{Y}$ given by 
\begin{equation}
\label{eq:ip}
y=\mathcal{G}(u) + \eta, \ \ \ \eta \sim \mathcal{N}(0,\Gamma).
\end{equation}
By allowing for a Bayesian approach one is interested in constructing, via an application of an infinite-dimensional Bayes' Theorem  \cite{AMS10}, a posterior measure of the random variable $u|y$
\begin{equation*}
\mu(du) = \frac{1}{Z}\exp(-\Phi(u;y))\mu_0(du),
\end{equation*}
with normalizing constant
\begin{equation*}
Z:= \int_{\mathcal{X}}\exp(-\Phi(u;y))\mu_0(du),
\end{equation*}
such that our data-likelihood is in the form of a potential
\begin{equation*}
\Phi(u;y) = \frac{1}{2} \| \Gamma^{-1/2}(y - \mathcal{G}(u)) \|^{2},
\end{equation*}
with the addition of a prior measure $\mu_0$.
This has been recently studied where there have been advancements in both computational and theoretical understanding \cite{BSW17, MAI16, ILS13, SS17}. From the computational aspect the EnKF was derived as a derivative-free inverse solver, which can be thought of as an optimizer which uses techniques from the Levenberg-Marquardt (LM) scheme \cite{MH97} combined with elements of the EnKF. It has been shown that applying these regularization techniques from LM \cite{MAI16,MH97} can improve the performance of the method. Regarding the theory of the EnKF for inverse problems, there has been progress on gaining analytical insight such as approximating continuous-time limits \cite{BSW17, SS17} within the context of inverse problems.  A new direction in this field which has emerged is the incorporation of hierarchical approaches for inverse problems \cite{ABPS14, DIS16, LLM17,RGLM16, RHL14}. In hierarchical inverse problems we are interested in recovering our unknown and a corresponding  hyperparameter $\theta \in \mathbb{R}^{+}$ that defines the unknown i.e. we wish to recover an unknown $(u,\theta) \in \mathcal{X} \times \mathbb{R}^{+}$ from noisy measurements $y$ where
\begin{equation*}
y=\mathcal{G}(u, \theta) + \eta.
\end{equation*} 
This allows for richer reconstructions as more information about the underlying unknown is available. An important feature of the EnKF applied to inverse problems is that it produces an ensemble of particles which lies within the linear span of the initial ensemble. This effect is known as the ``subspace property". By incorporating various hierarchical approaches we look to break this subspace property. This allows the solution to learn from information which may not be given within the span, but instead the data.  Specifically for EnKF inversion a hierarchical methodology was proposed in \cite{CIRS17} which demonstrated improvements over its non-hierarchical counterpart. The newly proposed method provides a way to effectively learn both the unknown and its hyperparameters that define it. This work used ideas from hierarchical computational statistics and applied it in an inverse problem setting \cite{PRS03, PRS07}. 

However regarding analytical results there has been no development in understanding these hierarchical approaches for the EnKF. This can be related to the lack of analysis on the EnKF. As of yet there has been work done on estimating non-hierarchical continuous-time limits  \cite{GMT11}. The purpose of this work is to build some analytical insight for hierarchical approaches that were used in \cite{CIRS17} for Bayesian inverse problems. It was shown in the linear noise-free case that one can attain a preconditioned gradient flow structure. Much of this will be based on extending the current theory in a hierarchical manner to the nonlinear noisy case, while providing an overview of the limit results attained in \cite{CIRS17}. We also aim to understand these approaches with modified versions of the EnKF, namely localization \cite{HM04} and covariance inflation \cite{AA99}. Both these techniques were developed to improve errors based on a small ensemble size. Similarly with some of the hierarchical approaches, localization and covariance inflation have the ability to break the subspace property. As a result it would be of interest to understand the limiting behaviour of these techniques. This includes conducting numerical experiments to verify hierarchical results obtained. We emphasize that with this work, rather than deriving new results for the EnKF, we aim to shed some light on hierarchical EnKF approaches for inverse problems and their respective continuous-time limits.
\subsection{Structure}
 The layout of this work is as follows; in Section \ref{section-2} we provide an overview of the EnKF applied to inverse problems. This will lead onto the formal derivation of the continuous-time limits applied to inverse problems. In Section \ref{section-3} we give a brief introduction for hierarchical approaches to EnKF inversion, while in Section \ref{section-4} we derive and present continuous-time limits for a list of variants on the EnKF. We verify these results through means of numerics in Section \ref{section-5}. Finally in Section \ref{section-6} we summarize our results and provide a brief mention on future work to consider.
 \subsection{Notation}
We assume that $(\mathcal{X}, \| \cdot \|, \langle \cdot \rangle)$ and $(\mathcal{Y}, \| \cdot \|, \langle \cdot \rangle)$ are two separable Hilbert spaces which are linked through the forward operator $\mathcal{G}: \mathcal{X} \rightarrow \mathcal{Y}$. The operator can be thought of as mapping from the space of parameters $\mathcal{X}$ to the observation space $\mathcal{Y}$. We denote the space of our hyperparameters as $\theta = (\sigma,\alpha,\ell) \in \mathbb{H}$ where $\mathbb{H}:= \mathbb{R} \times \mathbb{R}^+ \times \mathbb{R}^{+}$. For any such operator we define $ \langle \cdot, \cdot \rangle_{\Gamma} = \langle \Gamma^{-1/2}\cdot, \Gamma^{-1/2}\cdot \rangle$ and $\| \cdot \|_{\Gamma} = \| \Gamma^{-1/2} \cdot \|$, while for finite dimensions $| \cdot |_{\Gamma} = | \Gamma^{-1/2} \cdot|$ with $|\cdot|$ denoting Euclidean norm. $u^{(j)}_n$ will denote an ensemble of particles where $n$ is the iteration count and $j \in \{1,\ldots,J\}$ is the $j^{\textrm{th}}$ ensemble member.

\section{EnKF for inverse problems}\label{section-2}
The iterative EnKF method was first proposed in \cite{ILS13} to tackle  Bayesian inverse problems in a partial different equation (PDE)-constrained framework. The method can be derived as a sequential Monte-Carlo (SMC) approximation, where our probability measures of interest $\mu_n$ are defined by, for $h=N^{-1}$,
\begin{equation*}
\mu_n(du) \propto \exp(-nh\Phi(u;y))\mu_0(du),
\end{equation*}
thus leading to
\begin{equation*}
\mu_{n+1}(du) = \frac{1}{Z_n} \exp(-h\Phi(u;y))\mu_n(du),
\end{equation*}
where
\begin{equation*}
Z_n:= \int_{\mathcal{X}}\exp(-h\Phi(u;y))\mu_n(du).
\end{equation*}
We can construct our update for our probability measures $\mu_{n+1}$ through the operation
\begin{equation}
\label{eq:AD}
\mu_{n+1} = L_n \mu_n,
\end{equation}
where $L_n$ can be treated as a non-linear operator from $\mu_n$ to $\mu_{n+1}$ via an application of Bayes' Theorem. The idea behind the formulation of \eqref{eq:AD} is that it can be viewed as an artificial discrete-time dynamical system mapping the prior measure $\mu_0$ to the posterior measure $\mu_n$. Recall that with SMC methods one is interested in approximating a sequence of particles  and weights which take the form
\begin{equation*}
\mu_{n} \simeq \sum^{J}_{j=1}w^{(j)}_{n}\delta_{u^{(j)}_{n}}, \ \ \  j \in \{1,\ldots,J\},
\end{equation*}
with $\delta_{u^{(j)}_{n}}$ denoting the delta-Dirac mass at $u^{(j)}_{n}$. The weights associated with our sequence of particles satisfy the condition 
\begin{equation*}
\sum^{J}_{j=1}w^{(j)}_{n} = 1.
\end{equation*}
The SMC approach poses computational advantages over other Monte-Carlo methods, but still has limitations within it. These arise when the weights $\{w^{(j)}_n\}^J_{j=1}$ become degenerate i.e. that one of the weights becomes close to one where the rest are negligible \cite{APSS17}. The EnKF poses an improvement on this as its approximation has the form
\begin{equation*}
\mu_{n} \simeq \sum^{J}_{j=1} \delta_{u_{n}}^{(j)},
\end{equation*}
which excludes the weights. The EnKF for inverse problems, similarly to the EnKF, can be into two steps: a \emph{prediction step} and an \emph{update step}. The {prediction step} can be interpreted as mapping an ensemble of particles $u^{(j)}_n$ into the data space where we define our sample means for $J$ ensemble members
\begin{align*}
\bar{u} &= \frac{1}{J}\sum^{J}_{j=1} u^{(j)}_n, \\
\bar{\mathcal{G}} &= \frac{1}{J}\sum^{J}_{j=1} \mathcal{G}(u^{(j)}_{n}),
\end{align*}
and our empirical covariances
\begin{align}
\label{eq:up}
C^{up}_{n} &= \frac{1}{J}\sum^{J}_{k=1} (u^{(k)} - \bar{u}) \otimes (\mathcal{G}(u^{(k)}) - \bar{\mathcal{G}})^T   \\
\label{eq:pp}
C^{pp}_{n} &= \frac{1}{J}\sum^{J}_{k=1}  (\mathcal{G}(u^{(k)}) - \bar{\mathcal{G}})  \otimes  (\mathcal{G}(u^{(k)}) - \bar{\mathcal{G}})^T. 
\end{align}
The {update step} matches the mapped ensemble of particles to the data $y^{(j)}_{n+1}$ by using the calculated mean and covariances through the update formula
\begin{equation}
\label{eq:update}
u^{(j)}_{n+1} = u^{(j)}_n + C^{up}_{n} \big(C^{pp}_n + \Gamma \big)^{-1}\big(y^{(j)}_{n+1} - \mathcal{G}(u^{(j)}_n)\big),
\end{equation}
where
\begin{equation}
\label{eq:data}
y^{(j)}_{n+1} = y+ \iota^{(j)}_{n+1}, \ \ \iota^{(j)}_{n+1} \sim \mathcal{N}(0,h^{-1} \Gamma).
\end{equation}
\\
The EnKF for inverse problems possesses an important characteristic known as the \textit{subspace property} \cite{ILS13, LR09}. The property was first discussed \cite{LR09} which states that the updated ensemble of particles $u^{(j)}_{n+1}$ is preserved by the linear span of the initial ensemble $\mathcal{A} := \textrm{span}\{u^{(j)}_{0}\}$ for $j \in \{0,\ldots,J\}$. In the context of Gaussian priors, in the discrete case, this was proved in the following theorem. \\
\begin{theorem}
For every $(n,j) \in \mathbb{N} \times \{1,\ldots,J\}$ we have $u^{(j)}_{n+1} \in \mathcal{A}$ and hence $u_{n+1} \in \mathcal{A}$. \\
\end{theorem} 
\begin{proof}
The proof can be found in \cite{ILS13} by Iglesias et al.
\end{proof} \\\\
The property can be interpreted as given an initial ensemble, with particular set features depending how it is chosen, our solution to the inverse problem \eqref{eq:ip} will remain in the form that it is chosen initially. This can be advantageous if we know that the underlying unknown $u$ is of a similar form to the initial ensemble, where the converse of this is that it poses a limitation if they differ significantly. 
\subsection{Continuous-time limit}
\subsubsection{Nonlinear noisy case}
The continuous-time limit of the EnKF applied to inverse problems was considered in the work of Schillings et al. \cite{SS17}. We briefly recall the limit analysis here, firstly by considering the nonlinear noisy case. The limit here arises by taking the parameter $h \rightarrow 0$ from \eqref{eq:data}. We define $u_n = \{u^{(j)}_n\}^J_{j=1}$ and assume that $u_n \approx u(nh)$. Our update step \eqref{eq:update} can now be written in the form
\begin{align*}
u^{(j)}_{n+1} &= u^{(j)}_{n} + hC^{up}_n(u_n) \big(hC^{pp}_{n}(u_n) + \Gamma \big)^{-1} \big(y - \mathcal{G}(u^{(j)}_n)\big)\\ &+  hC^{up}_n(u_n) \big(hC^{pp}_{n}(u_n) + \Gamma \big)^{-1} \iota^{(j)}_{n+1} 
\\
&= u^{(j)}_{n} + hC^{up}_n(u_n) \big(hC^{pp}_{n}(u_n) + \Gamma \big)^{-1} \big(y - \mathcal{G}(u^{(j)}_n)\big) \\ &+  h^{\frac{1}{2}}C^{up}_n(u_n) \big(hC^{pp}_{n}(u_n) + \Gamma \big)^{-1}\sqrt{\Gamma}\zeta^{(j)}_{n+1},
\end{align*}
where $\zeta^{(j)}_{n+1} \sim \mathcal{N}(0,I)$. By taking the limit $h \rightarrow 0$, our limit can be viewed as a tamed Euler-Maruyama type discretization of the stochastic differential equations (SDEs)
\begin{equation*}
\frac{du^{(j)}}{dt} = C^{up}(u)\Gamma^{-1}\big(y - \mathcal{G}(u^{(j)})\big) + C^{up}(u)\sqrt{\Gamma^{-1}}\frac{dW^{(j)}}{dt},
\end{equation*}
with $W^{(j)}$ denoting independent cylindrical  Brownian motions. By substituting the form of the covariance operator \eqref{eq:up} we see
\begin{equation}
\label{eq:gen}
\frac{du^{(j)}}{dt} = \frac{1}{J}\sum^{J}_{k=1} \Big\langle \mathcal{G}(u^{(k)}) - \bar{\mathcal{G}},y- \mathcal{G}(u^{(j)}) + \sqrt{\Gamma}\frac{dW^{(j)}}{dt} \Big\rangle_{\Gamma}(u^{(k)} - \bar{u}).
\end{equation}
This derivation of the limit satisfies a generalization of the subspace property in continuous-time provided there is a solution to \eqref{eq:gen},  as the vector field is in the linear span of the ensemble. As we have just analyzed the limit in the noisy-case we will now turn our attention towards the linear noise-free case. 
\subsubsection{Linear noise-free case}
For this we take our forward operator $\mathcal{G}(\cdot) = A\cdot$ to be bounded and linear. Using this notion and by substituting our linear operator $A$ in \eqref{eq:gen} we have the following diffusion limit
\begin{equation}
\label{eq:linear}
\frac{du^{(j)}}{dt} = \frac{1}{J}\sum^{J}_{k=1} \Big\langle A(u^{(k)}- \bar{u}),y- Au^{(j)} + \sqrt{\Gamma}\frac{dW^{(j)}}{dt} \Big\rangle_{\Gamma}(u^{(k)} - \bar{u}).
\end{equation} 
By defining the empirical covariance operator 
\begin{equation*}
C(u) = \frac{1}{J}\sum^{J}_{k=1}(u^{(k)} - \bar{u}) \otimes (u^{(k)} - \bar{u})^T,	
\end{equation*}
and taking $\Gamma=0$ we can express \eqref{eq:linear} as 
\begin{equation}
\label{eq:gf-nh}
\frac{du^{(j)}}{dt} = -C(u) D_{u} \Phi(u^{(j)};y),
\end{equation}
with
\begin{equation*}
\Phi(u;y)= \frac{1}{2} \| \Gamma ^{-1/2}(y - Au)\|^2.
\end{equation*}
Thus we note that each particle performs a preconditioned gradient descent for $\Phi(\cdot;y)$ where all the gradient descents are preconditioned through the covariance $C(u)$. Since our covariance operator $C(u)$ is semi-positive definite we have that
\begin{equation*}
\frac{d}{dt}\Phi(u(t);y) = \frac{d}{dt}\frac{1}{2} \| \Gamma ^{-1/2}(y - Au)\|^2 \leq 0,
\end{equation*}
which provides a bound on $\|Au(t)\|_{\Gamma}$. In this case it was shown, through Theorem 2. in \cite{SS17}, that the gradient flow structure provides the existence of a solution satisfying the subspace property.
\section{Hierarchical ensemble Kalman inversion}\label{section-3}
In order to derive continuous-time limits we first recall a few properties of the hierarchical ensemble Kalman inversion (EKI). This will include newly defined update equations where we consider both the centred and non-centred approaches towards generating our prior measure $\mu_0$. Our prior $\mu_0 \sim \mathcal{N}(0,\mathcal{C})$ will be assumed to be of a Gaussian form with a Whittle-Mat\'{e}rn covariance function
\begin{equation}
\label{eq:fn}
c(x,x') = \sigma^2\frac{2^{1-\nu}}{\Gamma(\nu)} \bigg(\frac{|x-x'|}{\ell}\bigg)^{\nu}K_{\nu}\bigg(\frac{|x-x'|}{\ell}\bigg), \ \ \ \ x,x' \in \mathbb{R}^d,
\end{equation} 
where $K_{\nu}$ denotes a modified Bessel function of the second kind and $\Gamma(\nu)$ is a Gamma function. From \eqref{eq:fn} we also have the inclusion of three hyperparameters; the amplitude $\sigma \in \mathbb{R}$, the regularity $\nu = \alpha + d/2 \in \mathbb{R^{+}}$ and the length-scale $\ell  \in \mathbb{R}^+$. We can explicitly represent this covariance function through the following stochastic partial differential equation (SPDE), which is derived in \cite{RHL14},
\begin{equation}
\label{eq:SPDE}
(I-\ell^2 \triangle)^{\frac{\alpha}{2}}u = \ell^{d/2}\sqrt{\beta}\xi,
\end{equation}
 where $\xi  \in H^{-s}(D)$, $s>\frac{d}{2}$, for $D \subset \mathbb{R}^d$  is Gaussian white noise and
$$\beta=\sigma^2\frac{2^d\pi^{d/2}\Gamma(\alpha)}{\Gamma(\alpha-\frac{d}{2})}.$$
 Taking the SPDE defined above with $\beta \equiv 1$ we can rewrite \eqref{eq:SPDE} as
 \begin{equation}
\label{eq:easy}
\mathcal{C}^{-\frac{1}{2}}_{\theta}u = \xi,
\end{equation}
where $\theta = (\sigma,\alpha,\ell) \in \mathbb{H}$ denotes the collection of hyperparameters. The SPDE \eqref{eq:SPDE} is a common way of representing and expressing Gaussian random fields. This approach introduced by Lindgren et al.\ \cite{LRL11} was motivated to act as alternative to the Karhunen-Lo\`{e}ve expansion which posed significant computational benefits. They showed that the solution to the SPDE \eqref{eq:SPDE} omitted a covariance structure of the form \eqref{eq:fn}. Hierarchical modelling in statistics \cite{PRS07} has become quite crucial for better understanding of estimating the underlying unknown. 

This can be translated to inverse problems where we are not only interested in the field $u$ but its  hyperparameters associated with it. Within hierarchical modelling there are commonly two approaches one can take: the \emph{centred} approach and the \emph{non-centred} approach. These approaches were derived by Papaspiliopoulos et al. in \cite{PRS03, PRS07} in the context of Gaussian processes for computational statistics. Translating this to our inverse setting, the non-centred approach can be viewed as the parameterization under which we aim to solve $(\xi,\theta) \in H^{-s}(D) \times \mathbb{H}$ from \eqref{eq:easy}. While the centred approach differs as under its parameterization we aim to solve for $(u,\theta) \in \mathcal{X} \times \mathbb{H}$ from \eqref{eq:easy}. In terms of how the quantities $(u,\theta)$ and $(\xi,\theta)$ differ, their respective prior forms will be different as for the non-centred approach $\xi$ and $\theta$ are independent. Before discussing each approach in more detail we present an important proposition which states both approaches are equivilant when generating samples from \eqref{eq:SPDE}.  \\
\begin{proposition}
Given  a Gaussian random field $u$ with covariance operator $\mathcal{C}_{\theta}$, the centred and non-centred approaches to generate $u$ are equivalent. \\
\end{proposition}
\begin{proof}
Let $T:(\xi,\theta) \rightarrow u$ be a mapping where we choose $\mathcal{C}_{\theta}:=\ell^d \beta(I - \ell^2\Delta)^{-\alpha}$ for Equation \eqref{eq:easy}. We can express $u$ through the Karhunen-Lo\`{e}ve expansion
\begin{equation*}
 u=\sum_{k
 }\sqrt{\lambda_{k}}\hat{\xi_{k}}\phi_{k}, \ \ \ \ \ \ \hat{\xi_{k}} \sim \mathcal{N}(0,1),
 \end{equation*}
where $(\lambda^{2}_{k}, \phi_{k})$ are the eigenpairs of $\mathcal{C}_{\theta}$ for $k=1,2$.
Using the fact that both
\begin{align}
\label{eq:u}
u &=\sum_{k}\hat{u_{k}}\phi_{k},  \\
\label{eq:xi}
\xi &= \sum_{k}\hat{\xi_{k}}\phi_{k}, 
\end{align}
we see after substituting \eqref{eq:u} and \eqref{eq:xi} into \eqref{eq:easy}, where $k =  \left(
\begin{array}{c}
k_1\\
k_2\\
\end{array}
\right)
,$ that
\begin{align}
\frac{1}{\ell^{d/2}\sqrt{\beta}} (I - \ell^2|k|^2)^{\frac{\alpha}{2}}\sum_{k}\hat{u_{k}}\phi_{k}  &=  \sum_{k}\hat{\xi_{k}}\phi_{k}, \nonumber \\
\frac{1}{\ell^{d/2}\sqrt{\beta}}(I - \ell^2|k|^2)^{\frac{\alpha}{2}} \hat{u_{k}} &= \hat{\xi_{k}}. \nonumber
\end{align}
This implies 
\begin{equation*}
\hat{u_{k}}={\ell^{d/2}\sqrt{\beta}}(I-\ell^2 |k|^2)^{-\frac{\alpha}{2}} \hat{\xi_{k}},
\end{equation*}
\\
which is equivalent to $\lambda^2_k:= (I - \ell^2|k|^2)^{-\alpha}$.
\end{proof}

\subsection{Centred formulation}
\label{subsec:cen}
We now characterize our inverse problem through the centre formulation. For this approach our prior will have the form 
\begin{equation}
\label{eq:prior-c}
\mathbb{P}(u,\theta) = \mathbb{P}({u|\theta)} \mathbb{P}(\theta),
\end{equation} 
via the definition of conditional probability. We are interested in the recovery of our unknown $u \in \mathcal{X}$ from noisy measurements of our data $y$ where
\begin{equation}
\label{eq:inv-c}
y=\mathcal{G}(u)+\eta, \ \ \ \ \eta \sim \mathcal{N}(0,\Gamma). 
\end{equation}
We can further define a potential for our inverse problem $\Phi(u;y): \mathcal{X} \rightarrow \mathbb{R}$ where
\begin{equation}
\label{eq:phi-c}
\Phi(u;y) = \frac{1}{2}|y - \mathcal{G}(u)|^2_{\Gamma}.
\end{equation}
From the potential given in \eqref{eq:phi-c} we can define our data-likelihood as
\begin{equation}
\label{eq:like-c}
\mathbb{P}(y|u) = \exp \big(-\Phi(u;y)\big).
\end{equation}
Combing both our prior \eqref{eq:prior-c} and data-likelihood \eqref{eq:like-c}, via Bayes' Theorem, we can construct our posterior probability 
\newpage
\begin{align}
\mathbb{P}(u, \theta|y) &\propto \mathbb{P}(y|u)  \mathbb{P}(u, \theta) \nonumber \\
  &= \nonumber \exp \big(-\Phi(u;y)\big) \mathbb{P}(u|\theta) \mathbb{P}(\theta).
\end{align}
\begin{remark}
We note that the inverse problem associated with the centred approach \eqref{eq:inv-c} is the exact same as the non-hierarchical inverse problem \eqref{eq:ip} as the data does not depend on the updated hyperparameters. Thus in deriving continuous-time limits, the limit for our updated random field $u^{(j)}_n$ should be equivalent. \\\\
\end{remark}
As with the non-hierarchical method, we are interested in analyzing the hierarchical approaches influence on the subspace property, specifically whether they can break away from this property. With the centred approach we know that the data is only conditioned on the field $u$ and not its hyperparameters. Due to this we expect that with the centred approach, $(u, \theta)$ to lie within the span of the initial ensemble $\mathcal{A}$. The following theorem verifies this in the discrete case. \\
\begin{theorem}
\label{theorem_1}
For every $(n,j) \in \mathbb{N} \times \{1,\ldots,J\}$ we have $u^{(j)}_{n+1}$, $\theta^{(j)}_{n+1} \in \mathcal{A}$ and hence $u_{n+1}$, $\theta_{n+1} \in \mathcal{A}$. \\
\end{theorem}
\label{thm-c}
\begin{proof}
The proof follows similarly to that in \cite{ILS13} which is based on simple induction, but with the key difference of the inclusion of our hyperparameters $\theta^{(j)}_{n}$. We define our Kalman gain matrices as
\begin{align*}
    K^{u}_n &= \begin{pmatrix}
  C^{up}_n \big(C^{pp}_{n} +  \Gamma \big)^{-1} \\
  C^{u u}_n \big(C^{pp}_{n} +  \Gamma \big)^{-1}
    \end{pmatrix},
  \end{align*}
  \begin{align*}
    K^{\theta}_{n} &= \begin{pmatrix}
  C^{\theta p}_{n} \big(C^{pp}_{n} +  \Gamma \big)^{-1} \\
  C^{\theta \theta}_{n} \big(C^{pp}_{n} +  \Gamma \big)^{-1}
    \end{pmatrix},
  \end{align*}
  with empirical covariances $C^{up}_n, C^{uu}_n,  C^{\theta p}_n$. Recalling that the update equations are given as 
  \begin{align}
  \label{eq:xi_nc_u}
  u^{(j)}_{n+1} &= u^{(j)}_{n} + C^{up}_{n} (C^{pp}_{n} +  \Gamma)^{-1}(y^{(j)}_{n+1}-\mathcal{G}(u^{(j)}_n)), \\
  \label{eq:theta_nc_u}
\theta^{(j)}_{n+1} &= \theta^{(j)}_{n} + C^{\theta p}_{n} (C^{pp}_{n} + \Gamma)^{-1}(y^{(j)}_{n+1}-\mathcal{G}(u^{(j)}_n)).
  \end{align}
  By defining
  \begin{equation*}
  d^{(j)}_{n+1} = (C^{pp}_{n} + \Gamma)^{-1} (y^{(j)}_{n+1}-\mathcal{G}(u^{(j)}_n)),
  \end{equation*}
  Then the update formulas \eqref{eq:xi_nc_u} and \eqref{eq:theta_nc_u}  can be defined as 
  \begin{align*}
  u^{(j)}_{n+1} &= u^{(j)}_{n} + \frac{1}{J}\sum_{j=1}^{J}\langle \bar{\mathcal{G}}_{n+1}, d^{(j)}_{n+1} \rangle u^{(j)}_{n+1} \\
  &=  u^{(j)}_{n} + \frac{1}{J}\sum_{j=1}^{J}\langle \bar{\mathcal{G}}_{n+1}, d^{(j)}_{n+1} \rangle u^{(j)}_{n}, \\
  \theta^{(j)}_{n+1} &= \theta^{(j)}_{n} + \frac{1}{J}\sum_{j=1}^{J}\langle \bar{\mathcal{G}}_{n+1}, d^{(j)}_{n+1} \rangle \theta^{(j)}_{n+1} \\
  &=  \theta^{(j)}_{n} + \frac{1}{J}\sum_{j=1}^{J}\langle \bar{\mathcal{G}}_{n+1}, d^{(j)}_{n+1} \rangle \theta^{(j)}_{n}.
  \end{align*}
  At step size $n$ this shows that  $u^{(j)}_{n+1}, \theta^{(j)}_{n+1}  \in \mathcal{A}$ for $j \in \{1,\ldots,J\}.$ Hence since our outputs $u_{n+1},\theta_{n+1}$ at the end are defined as  
\begin{align*}
u_{n+1} &= \frac{1}{J} \sum^{J}_{j=1}u^{(j)}_{n+1}, \\
\theta_{n+1} &= \frac{1}{J} \sum^{J}_{j=1}\theta^{(j)}_{n+1},
\end{align*}
  it follows that both $u_{n+1}, \theta_{n+1} \in \mathcal{A}$.
\end{proof}

\subsection{Non-centred formulation}
As done previously in Subsection \ref{subsec:cen} we characterize our inverse problem but now for the non-centred formulation. For this approach our prior will have the form 
\begin{equation}
\label{eq:prior-nc}
\mathbb{P}(\xi,\theta) = \mathbb{P}(\xi)  \mathbb{P}(\theta),
\end{equation} 
via the definition of the non-centred approach in \cite{PRS03}. We are interested in the recovery of our unknown $(u,\theta) \in \mathcal{X} \times \mathbb{H}$ from noisy measurements of our data $y$ where
\begin{equation}
\label{eq:inv-nc}
y=\mathcal{G}(T(\xi,\theta))+\eta, \ \ \ \  \eta \sim \mathcal{N}(0,\Gamma),
\end{equation}
where $ T:(\xi,\theta) \rightarrow u$ is an operator such that $u = T(\xi,\theta)$. This modified formulation of our unknown arises from the SPDE \eqref{eq:SPDE}. As before we can further define a potential for our inverse problem $\Phi _{\mathrm{NC}}(\xi,\theta;y): \mathcal{X} \times \mathbb{H} \rightarrow \mathbb{R}$ where 
\begin{equation}
\label{eq:phi-nc}
\Phi _{\mathrm{NC}}(\xi,\theta;y) = \frac{1}{2}|y - \mathcal{G}(T(\xi,\theta))|^2_{\Gamma}.
\end{equation}
With $\textrm{NC}$ denoting non-centred. From the potential given in \eqref{eq:phi-nc} we can define our data-likelihood as
\begin{equation}
\label{eq:like-nc}
\mathbb{P}(y|\xi,\theta) = \exp \big(-\Phi _{\mathrm{NC}}(\xi,\theta;y)\big).
\end{equation}
Combing both our prior \eqref{eq:prior-nc} and data-likelihood \eqref{eq:like-nc}, via Bayes' Theorem, we can construct our posterior probability 
\begin{align}
\mathbb{P}(\xi, \theta|y) &\propto \mathbb{P}(y|\xi,\theta)  \mathbb{P}(\xi, \theta) \nonumber \\
  &= \nonumber \exp \big(-\Phi _{\mathrm{NC}}(\xi,\theta;y)\big)  \mathbb{P}(\xi)  \mathbb{P}(\theta). 
\end{align}
\begin{remark}
\label{remark-nc}
Unlike the centred approach, the non-centred formulation also differs as shown in the inverse problem \eqref{eq:inv-nc}, namely that the data it is dependent on both the field $u$ and the set of hyperparameters $\theta$ which is based on the transformation $T$. This would suggest the continuous-time limits would be different to the centred approach.
\end{remark} \\\\
The difference in the prior form between both approaches is important in understanding why the non-centred approach is advantageous. Given  that we are using $\xi$ and that it is independent on the initialization of $\theta$ in the prior form, and under the transformation $T$, this allows a much less restriction induced by the subspace property. As a result both $\xi$ and $\theta$ mix and update well, showcasing improvements over the centred approach. Also Numerics in \cite{CIRS17} demonstrated this for a range of non-linear PDE based inverse problems. The following theorem highlights this key difference related to the subspace property.
\\
\begin{theorem}
\label{nc_theorem}
For every $(n,j) \in \mathbb{N} \times \{1,\ldots,J\}$ we have $\xi^{(j)}_{n+1}$, $\theta^{(j)}_{n+1} \in \mathcal{A}$ and  $\xi_{n+1}$, $\theta_{n+1} \in \mathcal{A}$ hence $u_{n+1} \notin T\mathcal{A}$, where $T\mathcal{A}$ is the space containing the transformed ensemble of particles.
\end{theorem} \\
\begin{proof}
The proof follows very similarly to Theorem \ref{theorem_1} but with the difference of the transformation $T(\xi,\theta)=u$ which abides by a difference space than the one of the initial ensemble $\mathcal{A}$. Therefore $u_{n+1} \notin T\mathcal{A}$.
\end{proof}

\begin{table}[h]
\begin{tabular}{l | c |c}\label{tablecomparison}
 & \textbf{Centred approach} & \textbf{Non-centred approach}  \\\hline
Inverse problem & ${y} = \mathcal{G}({u}) + \eta$ &  ${y}  = \mathcal{G}(T({\xi,\theta})) + \eta$\\ \hline
Prior  &$\mu_0 \equiv \mathbb{P}({u,\theta})$ &  $\mu_0 \equiv \mathbb{P}({\xi,\theta})$ \\
&$\mu_0 \equiv \mathbb{P}({u}|{\theta}) \times \mathbb{P}({u})$ &  $\mu_0 \equiv \mathbb{P}({\xi}) \times \mathbb{P}({\theta})$\\\hline
Likelihood  &$\Phi({u};{y} ) = \frac{1}{2} | {y}  - \mathcal{G}({u})|^2_{\Gamma}$ & $\Phi_{\textrm{NC}}({\xi,\theta};{y} ) = \frac{1}{2} | {y}  - \mathcal{G}(T({\xi,\theta}))|^2_{\Gamma}$\\ 
 &$\mathbb{P}({y}|{u}) = e^{-\Phi({u};{y})}$ & $\mathbb{P}(y|{\xi,\theta}) = e^{-\Phi_{\textrm{NC}}({\xi,\theta};{y})}$\\ \hline
Posterior & $\mathbb{P}({u,\theta}|{y} ) \propto \mathbb{P}(y|{u}) \times \mathbb{P}({u})$ & $\mathbb{P}({\xi,\theta}|{y} ) \propto \mathbb{P}(y|{\xi,\theta}) \times \mathbb{P}({\xi})$\\
 & $\mathbb{P}({u,\theta}|{y} ) \propto  e^{-\Phi({u};{y} )} \mathbb{P}({u}|{\theta})  \mathbb{P}({\theta})$ & $\mathbb{P}({\xi,\theta}|{y} ) \propto  e^{-\Phi_{\textrm{NC}}({\xi,\theta};{y} )} \mathbb{P}({\xi}) \mathbb{P}({\theta})$
\end{tabular}
\label{table1}
\caption{Comparison of both hierarchical approaches.}
\end{table}
\section{Hierarchical continuous-time limits}\label{section-4}
\subsection{Centred approach}
\subsection{Nonlinear noisy case}
We begin our derivation of a continuous-limit for the hierarchical iterative EnKF method by considering firstly the centred approach. As we are interested now in $(u,\theta) \in \mathcal{X} \times \mathbb{H}$ we can construct a general posterior measure for $(u,\theta|y)$
\begin{equation*}
\mu(du,d \theta) = \frac{1}{Z}\exp(-\Phi(u;y))\mu_0(du,d\theta),
\end{equation*}
with
\begin{equation*}
Z: = \int_{\mathcal{X} \times \mathbb{H}}\exp(-\Phi(u;y))\mu_0(du,d\theta).
\end{equation*}
Similarly with the non-hierarchical EnKF, we can derive an approximation of the posterior measure through introducing an artificial dynamical system $\mu_{n+1} = L_n \mu_{n}$ where
\begin{equation*}
\mu_{n+1}(du,d\theta) = \frac{1}{Z_n} \exp(-h\Phi(u;y))\mu_n(du, d\theta),
\end{equation*}
and
\begin{equation*}
Z_{n}: = \int_{\mathcal{X} \times \mathbb{H}}\exp(-h\Phi(u;y))\mu_n(du,d\theta).
\end{equation*}
To construct our continuous-time limit we recall that the updates equations with the hierarchical iterative EnKF for 
\begin{align*}
u^{(j)}_{n+1} = u^{(j)}_{n} + C^{up}_n (C^{pp}_n + h^{-1} \Gamma)^{-1}(y^{(j)}_{n+1} - \mathcal{G}(u^{(j)}_n)) \\
\theta^{(j)}_{n+1} = \theta^{(j)}_{n} + C^{\theta p}_n (C^{pp}_n + h^{-1} \Gamma)^{-1}(y^{(j)}_{n+1} - \mathcal{G}(u^{(j)}_n)).
\end{align*}
Our update equations contain empirical covariance operators
\begin{align*}
C^{up}_{n} &= \sum^{J}_{k=1} (u^{(k)} - \bar{u}) \otimes (\mathcal{G}(u^{(k)}) - \bar{\mathcal{G}})  \\
C^{\theta p}_{n} &= \sum^{J}_{k=1} (\theta^{(k)} - \bar{\theta}) \otimes (\mathcal{G}(u^{(k)}) - \bar{\mathcal{G}})  \\
C^{pp}_{n} &= \sum^{J}_{k=1}  (\mathcal{G}(u^{(k)}) - \bar{\mathcal{G}})  \otimes  (\mathcal{G}(u^{(k)}) - \bar{\mathcal{G}}),
\end{align*}
where, as before, 
\begin{equation*}
\bar{\theta} = \frac{1}{J} \sum^{J}_{k=1} \theta^{(k)}_{n}, \ \ \ \ \bar{u} = \frac{1}{J} \sum^{J}_{k=1} u^{(k)}_{n}, \ \ \ \ \bar{\mathcal{G}} = \frac{1}{J} \sum^{J}_{k=1} \mathcal{G}(u^{(k)}_{n}),
\end{equation*}
for $j=1,\ldots,J$. We consider first the linear noise-free case where our forward operator takes the form $\mathcal{G}(\cdot) = A\cdot$ with $A \in \mathcal{L}((\mathcal{X} \times \mathbb{H}),\mathcal{Y})$. By taking the limit of our update equations as $h \rightarrow 0$ this leads to an Euler-Maruyama (EM) discretization of the form
\begin{align}
\label{eq:u1}
\frac{du^{(j)}}{dt} = C^{up}(u)\Gamma^{-1}\big(y - \mathcal{G}(u^{(j)})\big) + C^{up}(u)\sqrt{\Gamma^{-1}}\frac{dW^{(j)}}{dt} \\
\label{eq:u2}
\frac{d\theta^{(j)}}{dt} = C^{\theta p}(u)\Gamma^{-1}\big(y - \mathcal{G}(u^{(j)})\big) + C^{\theta p}(u)\sqrt{\Gamma^{-1}}\frac{dW^{(j)}}{dt},
\end{align}
such that $W^{(j)}$ are cylindrical  Brownian motions. By substituting the covariance operators $C^{up}_n, C^{\theta p}_n$ in \eqref{eq:u1} and \eqref{eq:u2} this leads to

\begin{align}
\label{eq:lim_c1}
\frac{du^{(j)}}{dt} = \frac{1}{J}\sum^{J}_{k=1} \Big\langle \mathcal{G}(u^{(k)}) - \bar{\mathcal{G}},y- \mathcal{G}(u^{(j)}) + \sqrt{\Gamma}\frac{dW^{(j)}}{dt} \Big\rangle_{\Gamma}(u^{(k)} - \bar{u}) \\ 
\label{eq:lim_c2}
\frac{d\theta^{(j)}}{dt} = \frac{1}{J}\sum^{J}_{k=1} \Big\langle \mathcal{G}(u^{(k)}) - \bar{\mathcal{G}},y- \mathcal{G}(u^{(j)}) + \sqrt{\Gamma}\frac{dW^{(j)}}{dt} \Big\rangle_{\Gamma}(\theta^{(k)} - \bar{\theta}).
\end{align}
In the hierarchical case the key distinguishment we see is firstly that our formulation of our measure differs as we take more than one underlying unknown, but also, when taking the limit $h \rightarrow 0$ we see we have coupled systems of SDEs. Using the same arguments in the non-hierarchical case given there is a solution to both \eqref{eq:lim_c1} and \eqref{eq:lim_c2}

\subsection{Linear noise-free case}
which after further substitution of the linear operator $A \in \mathcal{L}((\mathcal{X} \times \mathbb{H}),\mathcal{Y})$ our coupled SDEs read
\begin{align*}
\frac{du^{(j)}}{dt} &= \frac{1}{J}\sum^{J}_{k=1} \Big\langle A(u^{(k)}- \bar{u}),y- Au^{(j)} + \sqrt{\Gamma}\frac{dW^{(j)}}{dt} \Big\rangle_{\Gamma}(u^{(k)} - \bar{u}), \\
\frac{d\theta^{(j)}}{dt} &= \frac{1}{J}\sum^{J}_{k=1} \Big\langle A(\theta^{(k)}- \bar{\theta}),y- Au^{(j)} + \sqrt{\Gamma}\frac{dW^{(j)}}{dt} \Big\rangle_{\Gamma}(\theta^{(k)} - \bar{\theta}).
\end{align*}
Given our covariance operators for the centred approach
\begin{align}
\label{eq:cent-u}
C(u) &= \frac{1}{J}\sum^{J}_{k=1}(u^{(k)} - \bar{u}) \otimes (u^{(k)} - \bar{u}), \\
\label{eq:cent-l}
C(\theta) &= \frac{1}{J}\sum^{J}_{k=1}(\theta^{(k)} - \bar{\theta}) \otimes (\theta^{(k)} - \bar{\theta}),	
\end{align}
and $\Gamma=0$, we can express \eqref{eq:cent-u} and \eqref{eq:cent-l} as 
\begin{equation}
\label{eq:c-u}
\frac{du^{(j)}}{dt} = C(u) D_{u} \Phi(u^{(j)};y), \\
\end{equation}
where our potential is defined as
\begin{equation*}
\Phi(u;y)= \frac{1}{2} \| \Gamma ^{-1/2}(y - Au)\|^2.
\end{equation*}
As before we can interpret \eqref{eq:c-u} as each particle $\{u^{(j)}\}_{j=1}^{J}$ performing a gradient descent for $\Phi(\cdot;y)$. This is the exact same limit and gradient flow structure that we have in the non-hierarchical case \eqref{eq:gf-nh}. 
\subsection{Non-centred approach}
\subsection{Nonlinear noisy case}
Our construction of our posterior measure differs with the non-centred approach as we have a modified potential \eqref{eq:phi-nc}. Using this potential our posterior measure for $(\xi,\theta|y)$ now reads
\begin{equation*}
\mu(d\xi,d\theta) = \frac{1}{Z}\exp(-\Phi_{\textrm{NC}}((\xi,\theta);y))\mu_0(d\xi,d\theta),
\end{equation*}
with
\begin{equation*}
Z: = \int_{H^{-s}(D) \times \mathbb{H}}\exp(-\Phi_{\textrm{NC}}((\xi,\theta);y))\mu_0(d\xi,d\theta).
\end{equation*}
As similarly done for the centred approach we can derive an approximation by an artificial dynamical system $\mu_{n+1,\textrm{NC}} = L_{n,\textrm{NC}} \mu_{n,\textrm{NC}}$ where
\begin{equation*}
\mu_{n+1}(d\xi,d\theta) = \frac{1}{Z_n} \exp(-h\Phi_{\textrm{NC}}((\xi,\theta);y))\mu_{n,\textrm{NC}}(d\xi, d\theta),
\end{equation*}
and
\begin{equation*}
Z_{n}: = \int_{H^{-s}(D) \times \mathbb{H}}\exp(-h\Phi_{\textrm{NC}}((\xi,\theta);y))\mu_{n}(d\xi,d\theta).
\end{equation*}
The prediction step of the non-centred approach is a mirror to that of the centred approach but with the difference of updating $\xi$ instead of $u$, and we evaluate both$(\xi,\theta)$ in the forward evaluation. By defining $\mathcal{G}^{T} = \mathcal{G} \circ T$  our update equations for the non-centred approach are
\begin{align*}
\xi^{(j)}_{n+1} = \xi^{(j)}_{n} + C^{\xi p}_n (C^{pp}_n + h^{-1} \Gamma)^{-1}(y^{(j)}_{n+1} - \mathcal{G}^T(\xi^{(j)}_n,\theta^{(j)}_n)) \\
\theta^{(j)}_{n+1} = \theta^{(j)}_{n} + C^{\theta p}_n (C^{pp}_n + h^{-1} \Gamma)^{-1}(y^{(j)}_{n+1} - \mathcal{G}^T(\xi^{(j)}_n,\theta^{(j)}_n)),
\end{align*}
where we again assume that $\iota_{n+1} \sim \mathcal{N}(0,h^{-1}\Gamma)$ such that $y^{(j)}_{n+1} = y + \iota_{n+1}$, and that our empirical covariances are defined as 
\begin{align*}
C^{\xi p}_{n} &= \frac{1}{J}\sum^{J}_{k=1} (\xi^{(k)} - \bar{\xi}) \otimes (\mathcal{G}^T(\xi^{(k)},\theta^{(k)}) - \overline{\mathcal{G}^T}),  \\
C^{\theta p}_{n} &= \frac{1}{J} \sum^{J}_{k=1} (\theta^{(k)} - \bar{\theta}) \otimes (\mathcal{G}^T(\xi^{(k)},\theta^{(k)}) - \overline{\mathcal{G}^T}),  \\
C^{pp}_{n} &= \frac{1}{J} \sum^{J}_{k=1}  (\mathcal{G}^T(\xi^{(k)},\theta^{(k)}) - \bar{\mathcal{G}})  \otimes  (\mathcal{G}^T(\xi^{(k)},\theta^{(k)}) - \overline{\mathcal{G}^T}).
\end{align*}
We see that with the covariances defined above we have the addition of the  hyperparameter included in the evaluation of the forward operator which coincides with the inverse problem formulation \eqref{eq:inv-nc} where
\begin{equation*}
\overline{\mathcal{G}^T} = \frac{1}{J} \sum^{J}_{j=1}\mathcal{G}^{T}(\xi^{(j)}_n,\theta^{(j)}_n), \ \ \ j = 1,\ldots,J.
\end{equation*}
Therefore by taking the limit of our update equations as $h \rightarrow 0$, we have the coupled SDEs
\begin{align}
\label{eq:u1-nc}
\frac{d\xi^{(j)}}{dt} = C^{\xi p}(\cdot)\Gamma^{-1}\big(y - \mathcal{G}^T(\xi^{(j)},\theta^{(j)})\big) + C^{\xi p}(\cdot)\sqrt{\Gamma^{-1}}\frac{dW^{(j)}}{dt} \\
\label{eq:u2-nc}
\frac{d\theta^{(j)}}{dt} = C^{\theta p}(\cdot)\Gamma^{-1}\big(y - \mathcal{G}^T(\xi^{(j)},\theta^{(j)})\big)+ C^{\theta p}(\cdot)\sqrt{\Gamma^{-1}}\frac{dW^{(j)}}{dt},
\end{align}
such that $W^{(j)}$ are cylindrical Brownian motions. Using the formula for the covariances from \eqref{eq:u1-nc} and \eqref{eq:u2-nc}
\begin{align*}
\frac{du^{(j)}}{dt} &= \frac{1}{J}\sum^{J}_{k=1} \Big\langle  \mathcal{G}^T(\xi^{(k)},\theta^{(k)})- \bar{\mathcal{G}},y-  \mathcal{G}^T(\xi^{(j)},\theta^{(j)}) + \sqrt{\Gamma}\frac{dW^{(j)}}{dt} \Big\rangle_{\Gamma}(u^{(k)} - \bar{u}) \\
\frac{d\theta^{(j)}}{dt} &= \frac{1}{J}\sum^{J}_{k=1} \Big\langle  \mathcal{G}^T(\xi^{(j)},\theta^{(k)}) - \bar{\mathcal{G}},y-  \mathcal{G}^T(\xi^{(j)},\theta^{(j)}) + \sqrt{\Gamma}\frac{dW^{(j)}}{dt} \Big\rangle_{\Gamma}(\theta^{(k)} - \bar{\theta}).
\end{align*}
\subsection{Linear noise-free case}
As before we work in a linear setting where we define $\mathcal{G}^T(\cdot) = A\cdot$. Substituting $\mathcal{G}^T(\xi^{(k)},\theta^{(k)})=Au^{(k)}$, for $k=1,\ldots,J$, yields
\begin{align}
\label{eq:lim-nc-1}
\frac{d\xi^{(j)}}{dt} &= \frac{1}{J}\sum^{J}_{k=1} \Big\langle A(u^{(k)}- \bar{u}),y- Au^{(k)} + \sqrt{\Gamma}\frac{dW^{(j)}}{dt} \Big\rangle_{\Gamma} (\xi^{(k)} - \bar{\xi}) \\
\label{eq:lim-nc-2}
\frac{d\theta^{(j)}}{dt} &= \frac{1}{J}\sum^{J}_{k=1} \Big\langle A(u^{(k)} - \bar{u}),y- Au^{(k)}+ \sqrt{\Gamma}\frac{dW^{(j)}}{dt} \Big\rangle_{\Gamma} (\theta^{(k)} - \bar{\theta}) .
\end{align}
We notice with the SDEs the inclusion of the hyperparameter $\theta$  highlights one of the differences for the non-centred approach. Given our covariance operators for the non-centred approach
\begin{align*}
C(\xi) &= \frac{1}{J}\sum^{J}_{k=1}(\xi^{(k)} - \bar{\xi}) \otimes (\xi^{(k)} - \bar{\xi}) \\
C(\theta) &= \frac{1}{J}\sum^{J}_{k=1}(\theta^{(k)} - \bar{\theta}) \otimes (\theta^{(k)} - \bar{\theta}),	
\end{align*}
which we can express \eqref{eq:cent-u} and \eqref{eq:cent-l}, where $\Gamma = 0$, as
\begin{align} \label{eq:gf1}
\frac{d\xi^{(j)}}{dt} &= C(\xi) D_{u} \Phi_{\textrm{NC}}(u^{(j)};y) \\ \label{eq:gf2}
\frac{d\theta^{(j)}}{dt} &= C(\theta) D_{u} \Phi_{\textrm{NC}}(u^{(j)};y),
\end{align}
with potential
\begin{equation*}
\Phi_{\textrm{NC}}(\xi,\theta;y)= \frac{1}{2} \| \Gamma ^{-1/2}(y - Au)\|^2.
\end{equation*}
\\
For the non-centred approach we have derived a coupled gradient flow system for both the underlying unknown \eqref{eq:gf1} and the hyperparameters \eqref{eq:gf2} that differs from its centred counterpart. This is for the linear noisy case with continuous-time limits \eqref{eq:lim-nc-1} and \eqref{eq:lim-nc-2}.

\subsection{Hierarchical covariance inflation}
With the developments of the EnKF there has been considerate advancements which have looked at alternative approaches that provide improvements. An issue that can arise with the EnKF is rank deficiency. This problem occurs from the empirical covariances when the number of ensemble particles $J$ in the data space $\mathcal{Y}$ is less than that of the input space $\mathcal{X}$. One way to counteract this issue is through the technique of covariance inflation \cite{GE09}. We now aim to derive continuous-time limits of hierarchical covariance inflation, for EnkF inversion. We will do so specifically for the non-centred case, given its advantages we have discussed and shown in \cite{CIRS17}. This allows for a modification of  our covariances $C(\xi), C(\theta)$ given by
\begin{align}
\label{eq:vi_c}
C(\xi) \rightarrow \gamma C_0  + C(\xi) \\
\label{eq:vi_c2}
C(\theta) \rightarrow \gamma \theta_0  + C(\theta),
\end{align}
with $\gamma \in \mathbb{R}^{+}$. Substituting \eqref{eq:vi_c} and \eqref{eq:vi_c2}  in our gradient flow system leads to, for $j=1,\ldots,J$,
\begin{align*}
\frac{d\xi^{(j)}}{dt} &= (\gamma C_0  + C(\xi)) D_{u} \Phi_{\textrm{NC}}(u^{(j)};y) \\
\frac{d\theta^{(j)}}{dt} &= (\gamma \ell_0  + C(\theta))  D_{u} \Phi_{\textrm{NC}}(u^{(j)};y).
\end{align*}
By taking the inner product with $D_\xi \Phi_{\textrm{NC}}(u^{(j)};y)$ we have \\
\begin{align*}
\frac{d\Phi_{\textrm{NC}}(u^{(j)};y)}{dt} \leq - \gamma \| C^{1/2}_0 D_{u} \Phi_{\textrm{NC}}(u^{(j)};y)\|^2 \\
\frac{d\Phi_{\textrm{NC}}(u^{(j)};y)}{dt} \leq - \gamma \| \theta^{1/2}_0 D_{u} \Phi_{\textrm{NC}}(u^{(j)};y)\|^2.
\end{align*}
which indicates that all limits are contained in the critical points of both potentials.
\\
\subsection{Hierarchical localization} A further issue with the EnKF can arise from the correlation between the empirical covariances. If the correlation distance is long this can cause problems with updating our unknowns. {Localization} \cite{GE94} is a method that aids by cutting off these long distances which helps improve the update of the estimate. It is usually achieved through the aid of convolution kernels that reduce distances of distant regions.
The convolution kernels $\rho:D \times D \rightarrow \mathbb{R}$ are usually of the form
\begin{equation*}
\rho(x,y) = \exp\big(-(x-y)^{T} \big), 
\end{equation*}
given $D \subset \mathbb{R}^d$ for $d \in \mathbb{N}$, thereby allowing us to define continuous-time limits 
\begin{align*}
\frac{d\xi^{(j)}}{dt} &= C^{\mathrm{loc}}(\xi) D_{u} \Phi_{\textrm{NC}}(u^{(j)};y) \\
\frac{d\theta^{(j)}}{dt} &= C^{\mathrm{loc}}(\theta) D_{u} \Phi_{\textrm{NC}}(u^{(j)};y),
\end{align*}
where 
\begin{align*}
C^{\mathrm{loc}}(\xi) \Phi(x) = \int_{D} \phi(y)k(x,y)\rho(x,y)dy \\
C^{\mathrm{loc}}(\theta) \Phi(x) = \int_{D} \phi(y)k(x,y)\rho(x,y)dy,
\end{align*} 
given that $k(x,y)$ corresponds to the kernel of the covariances and $\phi \in \mathcal{X}$.

\section{Numerical experiments}\label{section-5}
We now wish to add some numerics to the theory discussed regarding the variants of localization and covariance inflation. We have seen through numerical investigation in \cite{CIRS17} that the theory discussed here matches with the results attained for various non-linear and linear inverse problems. In the context of this work we will only test for linear inverse problems, specifically a 1D elliptic PDE. Our numerics will consist of learning rates of hyperparameters and the reconstruction of the truth for both hierarchical localization and covariance inflation. Given a domain $D \subset \mathbb{R}^d$, for $d=1$, with Lipschitz boundary $\partial D$, our forward model is concerned with solving for $p \in H^1_0(D)$ from 
\begin{align}
\label{eq:fw}
\frac{d^2p}{dx^2} + p &= f \ \ \ x \in D, \\
\label{eq:bc}
p &= 0  \ \ \ x \in \partial D.
\end{align}
Here we assume a domain of $D = (0,\pi)$ with prescribed zero Dirichlet boundary conditions \eqref{eq:bc}. The inverse problem associated with the forward problem \eqref{eq:fw} is the recovery of noisy measurements from the right hand side $f$ where
\begin{equation}
\label{eq:inv_num1}
y_{j} = l_j(p) + \eta_j,
\end{equation}
such that $l_j \in V^*$ where $V^*$ is the dual space of $H^1_0(D)$. By defining $\mathcal{G}_j(T(\xi,\theta)) = l_j(p)$, where we take our unknown function $T(\xi,\theta)=f$, we can rewrite \eqref{eq:inv_num1}  as 
\begin{equation}
y = \mathcal{G}(T(\xi,\theta)) + \eta.
\end{equation}
Our inverse solver for our numerics will be the iterative ensemble Kalman method \cite{ILS13}, where we aim to reconstruct a Gaussian random field. As discussed in Section \ref{section-2} the algorithm can be split into two parts; the {prediction} step and the {update step}. Initially we set our initial ensemble based on a prior distribution. Our initial field will be set such that $\xi^{(j)}_0 \sim \mathcal{N}(0,\mathcal{C}_{\theta})$ where $\mathcal{C_{\theta}}$ takes the form \eqref{eq:easy}. To generate our initial ensemble with covariance structure of \eqref{eq:fn} we first discretize our SPDE \eqref{eq:SPDE} for $u$ using a 1D centred finite difference method
\begin{equation*}
u_i - \ell^2 \frac{u_{i+1} - 2u_i + u_{i-1}}{h^2} = \xi_i, \quad \xi_i \sim \mathcal{N}(0,\alpha \ell /h),
\end{equation*}
which in matrix form is given as
\[
  \begin{pmatrix}
    1+2\frac{\ell^2}{h^2} & -\frac{\ell^2}{h^2} & 0 & \hdots  & 0 \\
    -\frac{\ell^2}{h^2} &  1+2\frac{\ell^2}{h^2} & -\frac{\ell^2}{h^2} & \ddots & \vdots \\
      0 & -\frac{\ell^2}{h^2} & \ddots  & \ddots & 0  \\
     \vdots & \ddots & \ddots & \ddots & -\frac{\ell^2}{h^2} \\
    0 & \hdots & 0 & -\frac{\ell^2}{h^2} & 1+2\frac{\ell^2}{h^2} 
  \end{pmatrix}
  \begin{pmatrix}
  x_1 \\
  x_2 \\
  \vdots \\
  x_I
  \end{pmatrix}
  =
  \begin{pmatrix}
  \xi_1 \\
  \xi_2 \\
  \vdots \\
  \xi_I
  \end{pmatrix}.
\]
After generating $u$ we take our linear mapping $T:\mathcal{X} \rightarrow \mathcal{X}$ to generate samples of $\xi$. Our mesh size for our discretization is given as $h=1/50$ where $I=50$. From $\theta$ we will only treat the parameter of the length scale $\ell$ hierarchically. Our reason for this is that in a 1D numerical example the lengthscale has a more notable effect on how the input is generated. We keep $\sigma = 1$ and $\alpha=0.8$ while setting a prior on the lengthscale
\begin{equation}
\label{eq:ell}
\ell \sim \mathcal{U}[10,40].
\end{equation}
We generate our prior form $\mathbb{P}(\xi,\theta)$ by solving the SPDE \eqref{eq:SPDE} using a piecewise linear finite element method. Our truths will be chosen such that $\xi^{\dagger} \sim \mathcal{N}(0,\mathcal{C}^{\dagger}_{\theta})$, similar to the initial ensemble, where $\theta^{\dagger} = (\sigma^{\dagger},\alpha^{\dagger},\ell^{\dagger}) = (1,0.8,37)$. For our iterative method we set an ensemble size of $J=50$ and an iteration count of $n=15$, with covariance noise $\Gamma = 0.01^2I$. We discretize our PDE model \eqref{eq:inv_num1} with a different mesh size of $h^* = 1/50$ using a centred finite difference method. We make inference of our unknown through 16 chosen observations which lie on the true value of the unknown. For implementing covariance inflation we set the parameter as $\gamma = 0.1$.

In Figure \ref{fig:first} we analyze the performance of hierarchical localization by comparing it with non-hierarchical localization and the standard EnKF. We see that in the left subfigure
the standard EnKF and localization perform similarly emulating a smooth function. However for hierarchical localization we see an improved reconstruction which is more closely related to the truth, which incorporates its sharper features. This can be attributed to changes in the length scale which are verified in the right sub figure, where we see that by adopting a hierarchical approach we can effectively learn the true value of the length scale which is $\ell^{\dagger} = 37$. The learning of the lengthscale remains consistent with the results of \cite{CIRS17} where the hyperparameters learn the true value quickly and reach a limit before the learning stops prior to the termination of the experiment. We see similar results when analyzing hierarchical covariance inflation, where learning the length scale improves on the overall reconstruction of the truth as shown in Figure \ref{fig:second}.

\begin{figure}[h!]
\centering
\includegraphics[width=\linewidth]{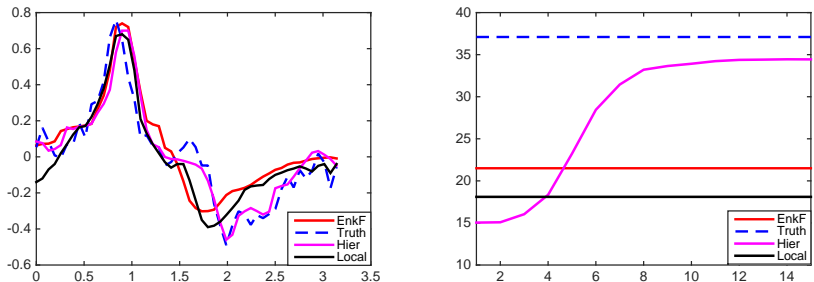}
 \caption{Performance of hierarchical localization. Left: reconstruction of the truth. Right: learning rate of the lengthscale.}
 \label{fig:first}
\end{figure}

\begin{figure}[h!]
\centering
\includegraphics[width=\linewidth]{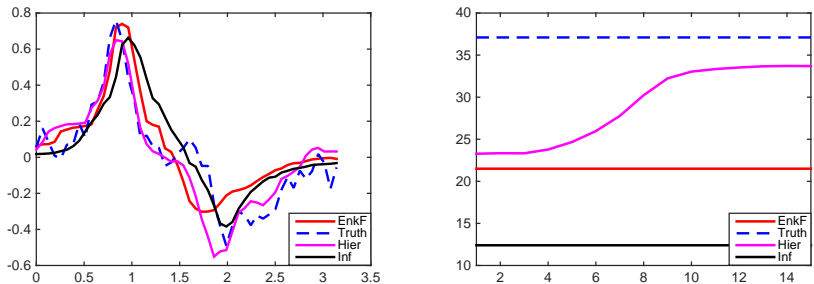}
 \caption{Performance of hierarchical covariance inflation. Left: reconstruction of the truth. Right: learning rate of the lengthscale.}
 \label{fig:second}
\end{figure}

\section{Conclusion} \label{section-6}
The objective of this work was to introduce analysis regarding the recent hierarchical approaches that were applied to EKI \cite{CIRS17}. We have given a detailed description and comparison of both the centred and non-centred approaches. For each case we have shown how they relate to the subspace property where we further derived continuous-time limits in in both the noisy and noise-free case.  Our analysis clarifies that by taking a non-centred approach one can significantly improve the performance of EKI. This is verified through the transformation which allows the ensemble of particles to leave the span of the initial ensemble. We introduced certain variants of the EnKF to show that hierarchically this can be achieved too, which was verified through a numerical experiment.

 One avenue of interest is to consider, as done in \cite{SS17}, the behaviour of the gradient flow structure defined for the non-centred approach \eqref{eq:gf1} and analyze the relationship with the subspace property. This is beyond the scope of this paper, but analyzing the behaviour could potentially result in improved convergence results over the non-hierarchical case.  A further direction is to extend this work by using certain SDE discretizations of EKI. This was analyzed in \cite{BSW17}, the natural extension of this would be to translate this in a hierarchical manner. Finally we provide a final remark that all of the analysis done thus far has been primarily in the linear case. This extends to both the EnKF in general and to EKI. The possibility of understanding limiting analysis in the non-linear case would provide much insight into the behaviour of the EnKF.
 \\\\
 \textbf{Acknowledgments} \ The author is grateful to Daniel Sanz-Alonso, Claudia Schillings and Andrew Stuart for fruitful and helpful discussions. The author was partially supported by the EPSRC MASDOC Graduate Training Program and by Premier Oil.

          \end{document}